\documentclass[preprint,12pt]{article}

\usepackage[english]{babel}
\usepackage{amsmath,amsthm,amssymb,scrextend}
\usepackage{graphicx}
\usepackage[colorlinks=true, allcolors=blue]{hyperref}
\usepackage{xargs}                     
\usepackage{xcolor}  
\usepackage{enumerate}
\usepackage[colorinlistoftodos,prependcaption,textsize=tiny]{todonotes}
\usepackage{tikz-cd} 
\usepackage[english]{babel}
\usepackage{natbib}
\usepackage{authblk}
\usepackage{setspace}
\usepackage{etoolbox}

\newcommandx{\comment}[2][1=]{\todo[linecolor=red,backgroundcolor=red!25,bordercolor=red,#1]{#2}}
\newcommand{\s}{\mathbb{S}}
\newcommand{\set}{\mathbf{Set_{3}}}
\newcommand{\ms}{\mathbf{Met_3}^{S}}
\newcommand{\mc}{\mathbf{Met_3}^{C}}
\newcommand{\ml}{\mathbf{Met_3}^{L}}
\newcommand{\nn}{\mathbb{N}}
\newcommand{\al}{\alpha}
\newcommand{\mtimes}{M\otimes -}
\newcommand{\bc}{\mathbf{C}}

\newtheorem{lemma}{Lemma}
\newtheorem{remark}{Remark}
\newtheorem{proposition}{Proposition}
\newtheorem{example}{Example}
\newtheorem{theorem}{Theorem}

\providecommand{\keywords}[1]
{
  \small	
  \textbf{\textit{Keywords--}} #1
}

\newcounter{bibcount}
\makeatletter
\patchcmd{\@lbibitem}{\item[}{\item[\hfil\stepcounter{bibcount}{\thebibcount.}}{}{}
\renewcommand\NAT@bibsetup%
   [1]{\setlength{\leftmargin}{\bibhang}\setlength{\itemindent}{-\parindent}%
       \setlength{\itemsep}{\bibsep}\setlength{\parsep}{\z@}}
\makeatother

\begin{document}

\title{Sierpinski Gasket as a Final Coalgebra Obtained by Cauchy Completing the Initial Algebra}
\author[1]{Jayampathy Ratnayake} 
\author[2]{Annanthakrishna Manokaran} 
\author[1]{Romaine Jayewardene} 
\affil[1]{Department of Mathematics, University of Colombo, Sri Lanka}
\affil[2]{Department of Mathematics and Statistics, University of Jaffna, Sri Lanka}
\date{}
\maketitle

\begin{abstract}
This paper presents the Sierpinski Gasket ($\s$) as a final coalgebra obtained by Cauchy completing the initial algebra for an endofunctor on the category of tri-pointed one bounded metric spaces with continuous maps. It has been previously observed that $\s$ is bi-Lipschitz equivalent to the coalgebra obtained by completing the initial algebra, where the latter was observed to be final when morphisms are restricted to short maps. This raised the question ``Is $\s$  the final coalgebra in the Lipschitz setting?''.  The results of this paper show that the natural setup is to consider all continuous functions. 

The description of the final coalgebra as the Cauchy completion of the initial algebra has been explicitly used to determine the mediating morphism from a given coalgebra to the the final coalgebra. This has been used to show that if the structure map of a coalgebra is continuous, then so is the mediating morphism. The description of $\s$ given here not only generalizes previous observations, but also unifies classical descriptions of $\s$. We also show, by means of an example, that $\s$ is not the final coalgebra if we consider only Lipschitz maps. 
\end{abstract}


\section{Introduction}\label{introduction}

The Sierpinski gasket ($\mathbb{S}$) can be described in many equivalent ways. One well known description is that it is the limit of the repeated removal of a triangle from an equilateral triangle. Instead of removing a triangle, one may also describe $\s$ as the limit of gluing triangles. Specifically, given an equilateral triangle (convex hull of three equidistant points), take three copies of it, shrink the size of each triangle by a factor of $\frac{1}{2}$ and then glue three copies so that each triangle touches the other triangles at a corner (see Figure \ref{ffig1} for a similar example with a different shape). Another description of $\s$ is based on a system of iterated functions. Consider the maps $\sigma_i:\mathbb{R}^2\rightarrow \mathbb{R}^2$, for $i\in\{a,b,c\}$, given by
\begin{align*}
    \sigma_a(x,y)&=(x/2, y/2)+(1/4, \sqrt{3}/4)\\
    \sigma_b(x,y)&=(x/2, y/2)\\
    \sigma_c(x,y)&=(x/2, y/2)+(1/2,0)
\end{align*}

\noindent Then, $\s$ is the unique nonempty compact set fixed by these functions and it satisfy the property $\s =\bigcup_{i\in\{a,b,c\}} \sigma_i\left(\s\right)$ . Moreover, we have the following characterization of $\s$ due to 
\cite{Hut}. Each $\sigma_i$, for $i=a,b,c$, has a unique fixed point, $(1/2, \sqrt{3}/2)$, $(0,0)$  and $(1,0)$ respectively. Then, $\s$ is the closure of the orbit of these three points with respect to the $\sigma_i$'s. From now on we use $\s$ to refer to this set.\\

Let $\mc$ be the category of tri-pointed one bounded metric spaces with the choice of morphisms as the continuous functions preserving distinguished points. In this paper, we describe $\s$ as a final coalgebra for an endofunctor $F$ on $\mc$, that is obtained by ``Cauchy completing" an initial algebra of $F$.\\

This paper was motivated by \cite{Bhat}, \cite{moss} and \cite{Lein}. An important question that \cite{Bhat} has tried to answer, with $\s$ as a motivating example, is whether a description of self similar objects as in \cite{Lein} can be given in the context of metric spaces that is also analogous to classical results such as those in \cite{Hut}. Some important observations have already been made in \cite{Bhat} (and also in \cite{moss}). Specifically it has been shown there that $\s$ is bi-Lipschitz isomorphic to the coalgebra that is final when the morphisms are restricted to short maps and that it is obtained by ``Cauchy completing" the initial algebra. However, as the observations in \cite{Bhat} on the finality of the completion of the initial algebra  limit to the short maps and $\s$ is only bi-Lipschitz isomorphic to this coalgebra, the description falls short of giving $\s$ as a final coalgebra. Moreover, these observations prompted the question whether $\s$ is the final coalgebra when the morphisms are chosen to be Lipschitz functions.\\

In this paper, we make explicit use of the description of the final coalgebra as the Cauchy completion of the initial algebra to determine the mediating morphisms. This enables the description of $\s$ as a final coalgebra in the context of continuous functions. We also show, by means of an example,  that $\s$ is not the final coalgebra if we  consider only Lipschitz maps. Thus, it turns out that continuous maps is the natural choice of morphisms. These results not only give a universal characterization of $\s$, strengthening  observations in \cite{Bhat} and answering some questions raised there, but also enables us to unify several different ways of describing $\s$. Before going in the the main discussion in the next section, we show below how the observations here unify different descriptions of $\s$ and also point out some analogous results in the literature.\\

The explicit description of the functor $F$ and $\s$ as an $F$-coalgebra, $\sigma : \s \rightarrow F(\s)$, will be given later (Example \ref{Sierpinski}). For the moment, with abuse of notation, we will refer to this coalgebra by the carrier set $\s$.\\

From the description of the functor and that $\s$ is a final $F$-coalgebra, $\s$ is realized as a self-similar object as in \cite{Lein}, obtained by gluing three copies of itself. Moreover, the initial algebra, consisting of ``finitely described'' points, can be compared (isomorphic in $\mc$) to the orbit of the three fixed points of the three contractions $\sigma_i$. The description of $\s$ as the Cauchy completion of the initial algebra is in analogy to Hutchinson's description (\cite{Hut}) of $\s$ as the closure of the orbit of fixed points of the iterations. The initial algebra can be considered as the ``finitely described'' points, whereas its completion, the final coalgebra, captures the ``limit''. \\

There is a close relationship between the set setting and the continuous setting. Results in \cite{Bhat} already describe $\s$ as a final coalgebra in the set setting. However, it follows that the initial algebra in the set setting can be given a natural metric and the Cauchy completion of the initial algebra with respect this metric is the final coalgebra in both the set and the metric settings. This can be compared with some of the observations in \cite{Lein}. It was shown in  \cite{Lein}, that for the functors considered there (which are defined in both the set and the topological settings), the final coalgebra in the set setting carries a natural topology, which makes it the final coalgebra in the topological setting. The analogous result  in our context considers metric spaces with continuous functions instead of topological spaces.\\

We believe that due to the above reasons, the description of $\s$ as a final coalgebra in the metric setting given here is more than yet another way of describing $\s$.  It not only unifies different ways of describing $\s$ and give $\s$ a universal characterization, but also a motivating example to develop a general theory for the case of metric fractals, that is similar to the work of \cite{Lein}. Some efforts to give similar examples can be found in \cite{Bhat, Bhatpadic, Mano} as well. This paper addresses how observations in these papers can be improved to be applicable in a general context.\\

We start by recalling some results in \cite{Bhat}. However, in several cases we give a different proof or a different emphasis to the results, to support the main ideas of the paper. It is noteworthy to say that to describe the final coalgebra as the Cauchy completion, one needs to first restrict $F$ to the subcategory $\ms$, where the morphisms are short maps, as the initial algebra is not known in the case of $\mc$ . Some important observations of this paper are that given any continuous $F$-coalgebra, the unique map to $\s$ can be defined canonically by using the fact that $\s$ is the Cauchy completion of the initial algebra and that this map is continuous (Proposition \ref{defOfMediatingMor} and Proposition \ref{contOff}).\\

In Section \ref{secOnPrilm}, we will outline the basic framework, including the definitions of the categories and the functor $F$ that are being considered here. We consider the set setting in Section \ref{secOnSET} and Section \ref{secOnSet}. The metric setting is considered in Section \ref{secOnMSC}. This is where we determine the final coalgebra of $F$ on $\mc$ and relate it to $\s$. Along the way we show how the mediating morphism from a coalgebra to the final coalgebra can be obtained by the limit of a sequence obtained by iterating the coalgebra. This construction relies on the fact that the final coalgebra is the Cauchy completion of the initial algebra.  In Section \ref{examples} we consider some examples and give counter examples to show that the Lipschitz setting is not sufficient. In particular, we show that the initial algebra of $F$ on $\ms$ is not the initial algebra of $F$ on $\mc$. We still do not know whether the initial algebra of $F$ on  $\mc$ exists.  We also show that the final coalgebra  and the initial algebra  of $F$  on $\ms$ are not the final coalgebra and the initial algebra of $F$ on $\ml$, subcategory of $\ms$ where the morphisms are Lipschitz maps.  In the case of $F$ on $\ml$, we do not know whether the final coalgebra or the initial algebra exist. 

\section{Preliminaries}
\label{secOnPrilm}

A tri-pointed set is a set having three distinguished elements. Such a set can be identified by a quadruple  $(X,T,L,R)$, consisting of a set $X$ and three distinguished elements $T, L$ and $R$ of $X$. When it is clear from the context, we often omit the distinguished points from the description of the set. We use subscripts to differentiate distinguished points of two tri-pointed sets $X$ and $Y$. For example, $T_X$ and $T_Y$ for $T$ of $X$ and of $Y$ respectively. There is the category of tri-pointed sets, $~\mathbf{Set_{3}}$, whose objects are tri-pointed sets and morphisms are functions which preserve the distinguished elements.\\ 

A tri-pointed metric space $(X, d)$ is a tri-pointed set $(X, T_X, L_X, R_X)$ equipped with a one-bounded metric $d$ (i.e.,  $d(x,y) \leq 1 \,\,\, \forall\, x,y\in X$), such that the distance between any pair of distinguished elements is 1. The class of tri-pointed metric spaces can be raised to categories $\ms$, $\ml$ and $\mc$, where the morphisms are respectively short maps, Lipschitz maps and continuous maps that preserve the distinguished elements. Note that the subscript $3$ emphasizes that the sets are equipped with three distinguished points and the superscript indicates the choice of morphisms. Note that  $\ml$ and $\ms$ are subcategories of $\mc$,  $\ms$ is a subcategory of $\ml$, and there is a forgetful functor from these categories ($\ms$, $\ml$ and $\mc$) to $\set$.\\

The functor that is central to the discussion of this paper is defined as follows. These definitions were motivated by \cite{Lein} and have been considered previously in \cite{Bhat} and \cite{moss}. We first define $F$ on $\set$.\\ 

Fix the set $M=\{a,b,c\}$. Given a tri-pointed set $X$, $M \times X$ consists of three copies of $X$ labeled by $a,\,b$ and $c$. Let $\sim$ be the equivalence relation on  $M \times X$ generated by the relations $(b,T)\sim (a,L)$, $(a,R)\sim (c,T)$ and $(c,L) \sim (b,R)$. Let $M \otimes X$ be the set of equivalence classes $M \times X/\sim$, with respect to $\sim$. We will denote the equivalence  class of an element $(m, x)\in M \times X$ by $m\otimes x$. The set $M\otimes X$ is lifted to a tri-pointed set by choosing $a\otimes T$, $b\otimes L$ and  $c\otimes R$ as $T_{M\otimes X}$, $L_{M\otimes X}$ and $R_{M\otimes X}$ respectively. We set  $F( X) = M \otimes X$.\\

This description identifies (glues) certain points of the three copies of $X$. For example the relation $(b,T)~\sim(a,L)$ identifies $L$ of copy $a$ with $T$ of copy $b$, and similarly for other relations, as shown in Figure \ref{ffig1} (see also \cite{Bhat}). Finally, the top element of the top triangle, left element of the left triangle and right element of the right triangle were taken as the three distinguished points, making it an object of $\set$.\\

\begin{figure}[h!]
\centering
\includegraphics[scale=0.5]{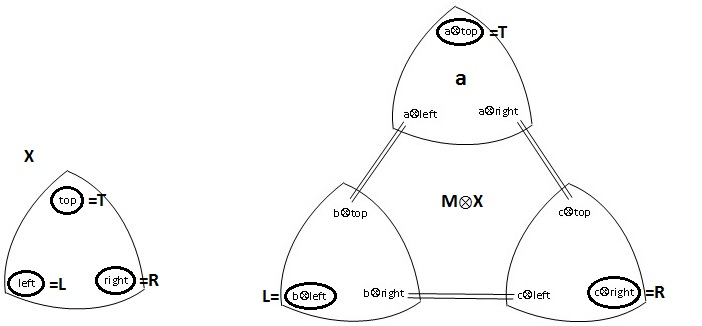}
\caption{Description of $F(X)=M\otimes X$}
\label{ffig1}
\end{figure}

The action of $F$ on morphisms is as follows. Given a morphism $f:X\rightarrow Y$ of tri-pointed sets, $F(f) : F(X)=M\otimes X\rightarrow  F(Y)=M\otimes Y$ is given by $F(m\otimes x)=m\otimes f(x)$. It is an easy exercise to show that $F(f)$ is well defined and preserves the distinguished elements. Hence, $F(f)$ is a morphism in $\set$.\\

Based on the description of $F$ on objects, if we may chose to write $F(f)= M\otimes f$. With this suggestion, $M\otimes -$ is an alternative notation for $F$. The definition of $\mtimes$ can be extended to define endofunctors on $\ms$,   $\ml$ and $\mc$, which are also designated by $F=\mtimes$. First, for a given tri-pointed metric space $(X,d)$, $M\times X$ is given the metric defined as follows.
	$$d\left((m,x),(n,y)\right)=\left\{%
	\begin{array}{ll}
	\frac{1}{2}d(x,y), & \hbox{$m=n$;} \\\\
	1, & \hbox{$m\neq n$.}
	\end{array}%
	\right.$$ 

We now consider the quotient metric on $M\otimes X$. Though the quotient metric in general is only a pseudometric, in our case it is indeed a metric and the distance between two elements can be computed explicitly as follows. 
	
  \begin{lemma}[Lemma 7, \cite{Bhat}]
  \label{LemmaMetricOnTensor}
  The distance $~d(m\otimes x, n\otimes y)~$ is given by the following formulas.  
\begin{align*}
d((a\otimes x),(a\otimes y) &=
	\frac{1}{2}d(x,y) \\
d((a\otimes x),(b\otimes y) &=\frac{1}{2} \min \{ d(x,L)+d(T,y)~,~d(x,R)+1+d(R,y) \}
\end{align*}

A formula similar to the last one holds for other cases of $a$, $b$ with $a \neq b$.
\end{lemma}

Given a tri-pointed metric space $(X, d)$, we use the same letter $d$ to identify the metrics defined above on $M\times X$ and $M\otimes X$, as it is understood from the context to which $d$ we refer. As $d$ on $X$ is one bounded, by definition, the  metrics on $M\times X$ and hence on $M\otimes X$ are also one bounded. Moreover, it can be easily verified that the distance between the distinguished elements of $M\otimes X$ is one. Hence $(M\otimes X,d)$  is an object of the category of tri-pointed metric spaces.\\

To define $F$ on the categories $\mc$ , $\ml$ and $\ms$, one needs to show that its action on functions preserves the property of being continuous, Lipschitz or short.

\begin{lemma}
\label{presFunctions}
Let $X$ and $Y$ be two tri-pointed metric spaces. If $f: X \rightarrow Y$  has any of the following properties then so does $F(f)=M\otimes f$.
\begin{enumerate}[i).]
\item \label{cnt} Continuous
\item \label{lip} Lipschitz
\item \label{shrt} Short map
\item \label{isoemb} Isometric embedding
\end{enumerate}
\end{lemma}

\begin{proof}
(\ref{cnt}). 
Let $\epsilon>0$ be arbitrary. Since $f$ is continuous, $\exists\,\delta >0$ such that $d(f(x),f(y))<\epsilon $,  whenever $d(x,y  )<\delta $.  Choose  $\delta_{0}=\min \{   
\frac {\delta}{2},
\frac {1}{4} \}$. Suppose $d_{M\otimes X}(a\otimes x, b\otimes y)<\delta_{0}$. Then we have $d_{X}(x,L_{X})< \delta$ and $d_{X}(y,T_{X})<\delta $. By the continuity of $f$ we have $d_{Y}(f(x),f(L_{X}))<\epsilon$ and $d_{Y}(f(y),f(T_{X}))<\epsilon$. Thus, $d_{M\otimes Y}(F(f)(a\otimes x), F(f) (b\otimes y) )<\epsilon$, which is the required condition for $F(f)$ to be continuous. This shows that $M\otimes f$ is continuous when $f$ is continuous.\\
(\ref{lip}).	Let $~f~$ be a Lipschitz continuous function with the Lipschitz constant $~k~$. Note that 
\begin{align*}
\text{min} & \left\{ d\left(f(x),L\right)+d\left(T,f(y)\right)~,~d\left(f(x),R)+1+d(R,f(y)\right) \right\} \\ 
   & =\ \min \left\{ d(f(x),f(L))+d(f(T),f(y))~,~d(f(x),f(R))+1+d(f(R),f(y)) \right\}\\
   & \leq  \min \left\{kd(x,L)+kd(T,y)~,~kd(x,R)+kd(T,L)+kd(R,y) \right\}\\
   & =  k \, \min \left\{d(x,L)+d(T,y)~,~d(x,R)+d(T,L)+d(R,y)\right\}
\end{align*}

\noindent Therefore, 
\[
d\left(F(f)(a\otimes x),F(f)(b\otimes y)\right) = d\left(a\otimes f(x), b\otimes f(y)\right) \leq k d\left(a\otimes x,b\otimes y\right) 
\]
Thus, $M\otimes f$ is Lipschitz, whenever $f$ is Lipschitz.\\ \hspace{2pt}\\
(\ref{shrt}) is proved in Proposition 8 of \cite{Bhat} and the proof of (\ref{isoemb}) is similar to that of (\ref{lip}).
\end{proof}

It now follows from Lemma \ref{presFunctions} that $F$ is a well defined endofunctor on $\mc$ and that it restricts to an endofunctor on the categories $\ms$  and  $\ml$. We use the same letter $F$ to refer to all these functors and specify the domain when there is ambiguity.\\     

An example that is central to our discussion is the presentation of $\s$ as a final coalgebra for $F$. We recall this example from \cite{Bhat}.

\begin{example}[Sierpinski Gasket] \label{Sierpinski} Recall that $\s$ is the unique nonempty compact subset of $\mathbb{R}^2$ such that $\s=\bigcup_{i\in\{a,b,c\}} \sigma_i\left(\s \right)$, where $\sigma_i$, for $i\in\{a,b,c\}=M$, are as in Section \ref{introduction}. We consider $\s$ as an object of $\set$, by designating  $(1/2, \sqrt{3}/2)$, $(0,0)$ and $(1,0)$ (three fixed points of $\sigma_a$, $\sigma_b$ and $\sigma_c$ respectively) as the three distinguished elements $T$, $L$ and $R$ respectively. Define the map $\tau : M\otimes \s \rightarrow \s$ by $\tau\left(m\otimes x \right)=\sigma_m\left(x \right)$. It follows that $\tau$ is a well defined bijection preserving the distinguished elements.  Then $\s$ carries a coalgebra structure for $F$ on $\set$, given by $\sigma = \tau^{-1} : \s \rightarrow M\otimes \s$.

Moreover, it follows that $\sigma$ is a short map and in particular Lipschitz continuous. Hence $\sigma : \s \rightarrow M\otimes \s$ is also a coalgebra for $F=\mtimes$ on $\mc$ as well as on $\ml$  and  $\ms$.
\end{example}

Given a tri-pointed set $X$ (or a metric space), we will denote $F\left(F\left(X\right)\right) = M\otimes \left(M\otimes X \right)$ by $M^2\otimes X$. Similarly, given a morphism $f$ in $\set$, $M^2\otimes f$ will stand for $F\left(F\left( f\right)\right)=M\otimes\left(M\otimes f\right)$. Like wise, for any integer $n\in \nn$,  the $n$-fold composition of $F$ will be denoted by $M^n\otimes -$. We set $M^0\otimes X =X$ and $M^0f=f$. A typical element of $M^n\otimes X$ will be denoted by $m_1\otimes \cdots \otimes m_n \otimes x$.\\

The following observation easily follows from Lemma \ref{LemmaMetricOnTensor}.

\begin{lemma}[Lemma 15, \cite{Bhat}]
\label{totalybound}
Let $X$ be a tri-pointed metric space (i.e., an object of $\mc$), $n\in \nn$ and $m_1, \cdots m_n \in M$. Let $\mathbf{m}=m_1\otimes m_2 \otimes \cdots \times m_n$

For all $x_1, x_2 \in X$
\[
d\left(\mathbf{m} \otimes x_1 ~,~  \mathbf{m} \otimes x_2 \right) \leq 2^{-n} 
\]
In particular, for all $x_1$, $x_2$, $x_1'$, $x_2'\in X$
\[
\left| d\left( \mathbf{m} \otimes x_1~,~ \mathbf{m} \otimes x_2 \right) - d\left( \mathbf{m} \otimes x_1'~,~ \mathbf{m} \otimes x_2' \right) \right| \leq 2^{1-n}
\]
\end{lemma}


\section{Initial algebra of $F$ on $~\mathbf{Set_{3}}~$ and its completion }
\label{secOnSET}

The set $I=\{T, L, R\}$, with distinguished points as suggested by the notation, is the initial object of the category $\set$. We will first consider the initial algebra of $F$ and then describe a canonical metric on it and its Cauchy completion.

\subsection{Initial algebra $G$ of $F$ on $\set$}

One can easily determine the initial algebra for $F$ on $\set$ by considering the initial chain (see \cite{Bar} and \cite{Adamek}). The initial chain, determined by the initial object is
\begin{equation}
\label{initialChain}
\begin{tikzcd}
I \arrow{r}{!} & M\otimes I \arrow{r}{M\otimes !} & M^{2}\otimes I \arrow{r}{M^{2}\otimes !}  & \cdots \rightarrow M^{j-1}\otimes I \arrow{r}{M^{j-1}\otimes ! } &  M^{j}\otimes I \cdots 
\end{tikzcd}     
\end{equation}

\noindent  where $!$ is the unique morphism from $I$ to $M\otimes I$ in $\set$. \\

The colimit of the sequence can be determined as follows. Take $\,C=\bigcup M^{n}\otimes I$. Define a relation on $C$ as follows. Let $x,y \in C$. Then$~x\in M^{r}\otimes I $ and $y\in M^{s}\otimes I ~$  for some $~r,s$. Without lost of generality take $s>r$. The relation $\approx$ is defined by $x \approx y~$ iff $f(x)=y$; where $f=M^{s-1}\otimes !\circ \cdots  \circ M^{r}\otimes !$. Let $\sim$ be the equivalence closure of $\approx$. Then $G=C / \sim$, the quotient of $C$ by the equivalence relation $\sim$ is the colimit of the initial chain and also the carrier set of the initial algebra. For $x\in C$, denote by $[x]$ the equivalence class of $x$. $G$ is lifted to a tri-pointed set by declaring the equivalence classes $[T]$, $[L]$ and $[R]$ as the distinguished points $T_G$, $L_G$ and $R_G$ respectively. Morphisms to the colimit $C_{n}:M^{n}\otimes I \longrightarrow G$ are given by $C_{n}(x)=[x]$. The following lemma follows from Ad\a'amek's Theorem (\cite{Adamek}).
\begin{lemma}
\label{InitialAlgOnSet}
The initial algebra of $F$ on $\set$ is $(G,g:M\otimes G\rightarrow G)$, where $g:M\otimes G\rightarrow G$ is given by $g(m\otimes [x])=[m\otimes x]$.
\end{lemma} 

One can iterate an $F$-algebra $(X,e)$ to obtain the following chain.
\begin{equation}
\nonumber
\label{algebraIterated}
 \begin{tikzcd}
\cdots\arrow{r} M^{k}\otimes X \arrow{r}{M^{k-1}\otimes e} & M^{k-1}\otimes X \arrow{r} & \cdots  \arrow{r} & M\otimes X \arrow{r}{e} & X
\end{tikzcd}    
\end{equation}

\noindent For $s > r\geq 0$, we let $e_{rs}:M^{s}\otimes X \rightarrow M^{r}\otimes X$ be the composition of the maps $ M^{r}\otimes e$, $M^{r+1}\otimes e$,...,$M^{s-2}\otimes e$ and $M^{s-1}\otimes e$.
\begin{equation}
\nonumber
\label{esr}
e_{rs}=M^{r}\otimes e ~\circ~ M^{r+1}\otimes e ~\circ~ \cdots ~\circ~ M^{s-2}\otimes e ~\circ~ M^{s-1}\otimes e
\end{equation}

\noindent In particular, $e_{(r-1)r}=M^{r-1}\otimes e$. When $r=0$, $e_{0s}$ will be denoted simply by $e_s$. Thus, $e_s:M^s\otimes X\rightarrow X$ is the composition of the maps $e, M\otimes e$,...,$M^{s-1}\otimes e$.
\begin{equation}
\label{ek}
e_{s}=e\circ M\otimes e \circ \cdots  \circ M^{s-1}\otimes e
\end{equation}  

\noindent For $t>s>r\geq 0$, we have $e_{rt}=e_{rs}\circ e_{st}$.\\

The mediating morphism from the initial algebra $(G, g)$ can be computed explicitly as follows. Let $(X,e)$ be an algebra for $F$ on $\mathbf{Set_3}$, where $\,e\,$ is a function that preserves the distinguished points. For $x=[(m_{1}\otimes \cdots m_{k}\otimes d)]\in G$, where $d\in \{T,L,R\}$, let $\overline{x}=m_{1}\otimes \cdots m_{k}\otimes d_{X}$, where $d_{X}$ is the corresponding distinguished element of $X$. For instance if $d=T$, then $d_{X}=T_{X}$.

 \begin{lemma}
 \label{mediationMorAlgebra}
Let $\,(X,e)\,$ be an $F$-algebra on $\set$. Then, the mediating morphism from $G$ to $X$, $f:G \longrightarrow X$,  is given by  
$$f\left(\left[\left(m_{1}\otimes \cdots m_{k}\otimes d\right)\right]\right)=e_k\left( m_{1}\otimes \cdots m_{k}\otimes d_X\right).$$
\end{lemma} 

\begin{proof}
Let us first prove that $~f~$ is well defined. Let $[(p_{1}\otimes\cdots  p_{r}\otimes d_{r})]=[(q_{1}\otimes\cdots  q_{s}\otimes d_{s})]$. Thus we have $p_{1}\otimes\cdots  p_{r}\otimes d_{r}\sim q_{1}\otimes\cdots  q_{s}\otimes d_{s}$. As the case $r=s$ is trivial, without loss of generality we may assume $s>r$. Now, consider the map $!_{sr}:M^r\otimes I \rightarrow M^s\otimes I$, given by composition  $$!_{sr}= M^{s-1}\otimes !\circ M^{s-2}\otimes !\cdots M^{r+1}\otimes !\circ M^{r}\otimes !$$  of the maps in the initial chain (\ref{initialChain}). From the definition of $\sim$, it follows that $!_{sr} \left( p_{1}\otimes\cdots  p_{r}\otimes d_{r} \right)=  q_{1}\cdots  \otimes d_{s}$.\\

\noindent Since $I$ is the initial object the unique map $eta:I\rightarrow X$ is such that the following square commutes.
\begin{equation*}
\label{square1}
\begin{tikzcd}
M\otimes X  \arrow{r}{e} &  X\\
M\otimes I \arrow{u}{M\otimes  \eta}  & \arrow{l}{!} I \arrow{u}{\eta}
\end{tikzcd}     
\end{equation*}

By applying $F=M\otimes -$ to the above square repeatedly and stacking the resulting squares, we obtain the following commuting square. 
\begin{equation*}
\label{square2}
\begin{tikzcd}[column sep=huge]
M^{s}\otimes X \arrow{rrrr}{e_{rs}=M^{r}\otimes e ~\circ~ M^{r+1}\otimes e ~\circ~ \cdots ~\circ~ M^{s-2}\otimes e ~\circ~ M^{s-1}\otimes e~} & & & & M^{r}\otimes X\\
M^{s}\otimes I \arrow{u}{M^{s}\otimes  \eta}   & & & &   \arrow{llll}{!_{sr}=M^{s-1}\otimes !~\circ~ M^{s-2}\otimes !~\circ~\cdots ~\circ~ M^{r+1}\otimes !~\circ~ M^{r}\otimes ! } \arrow{u}{M^{r}\otimes  \eta} M^{r}\otimes I
\end{tikzcd}     
\end{equation*}

\noindent It now follows from the commutativity of the above diagram that $$\left(e_{rs} \right)\left(q_{1}\otimes\cdots  \otimes q_{s}\otimes d_{s_{X}}\right)=p_{1}\otimes\cdots \otimes  p_{r}\otimes d_{r_{X}}.$$ Because $e_s = e_r \circ e_{rs}$, we have; 
\begin{align*}
    f\left( [(q_{1}\otimes\cdots  q_{s}\otimes d_{s})] \right) &=  e_s\left( q_{1}\otimes\cdots  q_{s}\otimes d_{s_X} \right)\\
    &=\left( e_r\circ e_{rs} \right) \left( [(q_{1}\otimes\cdots  q_{s}\otimes d_{{s}_X})] \right)\\
    &=e_r\left( p_{1}\otimes\cdots  p_{r}\otimes d_{r_{X}} \right)\\
    &=f\left( [(p_{1}\otimes\cdots  p_{r}\otimes d_r )]\right).
\end{align*}   

\noindent Therefore, $f$ is well defined.\\

We are left to show that the following  diagram commutes. 
\begin{equation*}
\label{square3}
\begin{tikzcd}[column sep=huge]
M\otimes G \arrow{r}{g} \arrow{d}{M\otimes f}&  G \arrow{d}{f}\\
M\otimes X \arrow{r}{e} &  X
\end{tikzcd}     
\end{equation*}

\noindent This amounts to verifying the identity
\[
e\left(  m_{0}\otimes e_k \left( m_{1}\otimes m_{2}\cdots\otimes m_{k}\otimes d_{k_X} \right) \right) = e_{k+1}\left( m_{0}\otimes m_{1}\cdots  \otimes m_{k}\otimes d_{k_{X}} \right).
\]
\noindent i.e., we have to show that $e \circ \left( M\otimes e_k \right) = e_{k+1}$. Note that the equality follows by substituting the definition of $e_k$ to the left hand side and using the functoriality of $M\otimes -$.
\end{proof}

\subsection{Canonical metric on $G$}

The discrete metric is the canonical metric on  the initial object $I$ of $\set$ by which it is also an object of the category of tri-pointed metric spaces ($\mc$). For every $k\in \mathbb{N}$,  $M^k\otimes I$ also carries a canonical metric; namely, the quotient metric guaranteed by Lemma \ref{LemmaMetricOnTensor}. Moreover, $! : I \rightarrow M\otimes I$ is an isometric embedding as it maps the elements of $I$ to the corresponding distinguished elements. It then follows from  (\ref{isoemb}) of Lemma \ref{presFunctions}, that $M^k \otimes !$, for every $k\in \mathbb{N}$, is also an isometric embedding. Thus,the initial chain (\ref{initialChain}) is a diagram in $\mc$, where all the maps are isometric embeddings. In this case $G$, the colimit of the initial chain (Equation \ref{initialChain}) in $\set$, carries a canonical metric for which it is also the colimit of the initial chain in $\mc$.\\

Loosely speaking, $G$ can be understood as the union of $M^{k}\otimes I$ where $M^{k}\otimes I$ is realized as a subset of $M^{k+1}\otimes I $ via the isometric embedding $M^{k}\otimes !$. Formally, it is the quotient metric on $\,G= \left( \bigcup M^{n}\otimes I  \right) / \sim$. We will denote this metric by $d_G$ or simply by $d$ when it is understood from the context.\\

In particular, $C_{n}:M^{n}\otimes I \longrightarrow G$ is an isometric embedding, for every $n$. Given, $\chi_p = m_{1}\otimes \cdots m_{p} \otimes d_{p}\in M^p\otimes I$ and $\chi_q=m_{1}\otimes  \cdots \otimes m_{q} \otimes d_{q}\in M^q \otimes I$, the distance between $[\chi_p]$ and $[\chi_q]$ in $G$ can be computed as follows. Without loss of generality assume $p\geq q$. Then 
\begin{equation}
\label{metricG}
    d_G\left( [\chi_p] ,[\chi_q] \right) = d_q\left( !_{pq}\left(\chi_q\right) , \chi_p \right),
\end{equation} 

\noindent where $d_q$ is the metric on $M^q \otimes I$.\\

With this canonical metric on $G$, $M\otimes G$ is also equipped with the canonical metric induced by $G$, as guaranteed by Lemma \ref{LemmaMetricOnTensor}. From now on, when we refer to the metric structure of $G$, we shall always mean this metric on $G$ and $M\otimes G$. Moreover, it is clear from the definition, which also follows from Lambek's Lemma (\cite{Lam}), that $g$ is onto. The following lemma now  follows from the above observations and has already been observed in Proposition 10 of \cite{Bhat}.

\begin{lemma}[Proposition 10, \cite{Bhat} ]
\label{GIsometry}
The structure map $g:M\otimes G \rightarrow G$ of the initial algebra of $F$ on $\set$ is an isometric isomorphism in $\mc$.

In particular $(G, g)$ is an algebra for $F$ on $\mc$.
\end{lemma}

\begin{example}
\label{ExampleInitialAlgeMSCtri}
One may wonder whether $(G, g)$ is also the initial algebra for $F$ on $\mc$. We give an example to show that this is not the case. However, it should be recalled here that $(G, g)$ is the initial algebra for $F$ when the morphisms are restricted to be short maps, i.e., $F$ on $\ms$ (Theorem 11 of \cite{Bhat}). \\

Consider the tri-pointed metric space $Y=\{t, l, r\}$ equipped with the discrete metric (which is the only possible metric in this case) and distinguished points as suggested by the notation. Let the function $\al :M\otimes Y \longrightarrow  Y$ be given by $\al (a\otimes t)=t\,,\,\al (c\otimes r)=r$ and $\al (a\otimes l)=\al (a\otimes r)=\al (b\otimes l)=\al (b\otimes r)=l$. Clearly $\al$ is a morphism in $\set$. Because $\al$ is a function from a finite metric space, it is Lipschitz  and hence  continuous. Thus $(Y, \al)$ is an algebra for $F$ on $\mc$ (and also on $\set$).\\

Because $(G, g)$ is the initial algebra for $F$ on $\set$, there is a unique function $f_Y:G\rightarrow Y$ such that $f_Y \circ g = \al \circ (M\otimes f_Y)$, given by Lemma \ref{mediationMorAlgebra}. Note that in $G$, $[T]=[a\otimes T]=[a\otimes a\otimes T]=\cdots=[a\otimes a\otimes \cdots \otimes a \otimes T]=\cdots$,  and $f_Y([T])=f_Y([a\otimes a\otimes \cdots \otimes a \otimes T])=t$ as the distinguished elements are preserved by $f_Y$. Let $y=[a\otimes a\otimes \cdots \otimes a \otimes L]$, where $a$ occurs $n$ times. Then, it can be easily verified that $f_Y(y)=l$.\\

Note that $d([T],y)=\frac{1}{2^{n}}$; however,  $d(f_Y(T),f_Y(y))=d(t,l)=1$. Thus $f_Y$ is not continuous. Hence $(G,g)$ is not the initial algebra for $F$ on  $\mc$. This also shows that $(G, g)$ is not the initial algebra for $F$ on $\ml$.
\end{example}

\subsection{Cauchy completion of $G$}

As it has been shown in \cite{Bhat}, Cauchy completion can be described as a functor that commutes with $M\otimes -$. Given $(X,d)\in \mc$, we will denote the Cauchy completion of $X$ by $\mathbf{C}\left(X\right)$. We recall the following.\\

$\mathbf{C}\left(X\right)$ consists of the equivalence classes of Cauchy sequences in $X$, where two sequences $(a_{n})$ and $(b_{n})$ are equivalent if $\lim_{n\rightarrow \infty}d(a_{n},b_{n})=0$. Let us denote the equivalence class of the sequence $(a_{n})$ by $\left[a_{n}\right]$. The distance in $\mathbf{C}\left(X\right)$ is given by $d\left(\left[ a_n \right] \, , \, \left[ b_n \right] \right)=\lim_{n\rightarrow \infty} d\left(a_n , b_n\right)$. $\mathbf{C}\left(X\right)$ is a one bounded complete metric space. It can be made into a tri-pointed metric space by choosing the distinguished elements to be the constant sequences given by the corresponding distinguished elements of $X$. This construction extends to a functor $\mathbf{C}$ on the category of tri-pointed metric spaces, where for a continuous function $f:X\rightarrow Y$, $\mathbf{C}\left(f\right)\left(\left[x_n\right]\right)=\left[ \left( f(x_n) \right)\right]$. Thus the Cauchy completion $\bc$ is a functor $\mathbf{C}:\mc \rightarrow \mc$, where the image is actually the complete ones. It directly follows from the construction that if $f$ is a short map, then so is $\mathbf{C}\left(f\right)$. Thus the Cauchy completion also gives rise to an endofunctor $\mathbf{C}:\ms\rightarrow\ms$.\\

The map $\nu_X:X\rightarrow \mathbf{C}\left(X\right)$ that send each $x\in X$ to the equivalence class $[x]$ of the constant sequence $(x_n=x)$ is an isometric embedding of $X$ into a dense subset of $\mathbf{C}\left(X\right)$. As $\nu_X$ preserve the distinguished points, it is a morphism in $\mc$. This enables us to identify $X$ as a dense subset of $\bc(X)$.\\

$\bc$ can be composed with $F=\mtimes$ in two different ways: $\mathbf{C}\circ F$ and $F \circ \mathbf{C}$. However, these two compositions are naturally isomorphic. Specifically, $\beta_X : M\otimes \mathbf{C}(X) \rightarrow \mathbf{C}\left(M\otimes X\right)$ given by 
\[
\beta_X \left(m\otimes\left[x_{n}\right]\right)=\left[m\otimes x_{n}\right]
\]

is a natural isomorphism between $\bc\left(M\otimes X\right)$ and $M\otimes \bc\left(X\right)$. The inverse of $\beta_X$ can be described as follows. Consider $\bc\left(M\otimes X\right)$. Because $M$ is finite, in every sequence $[m_n\otimes x_n]\in \bc\left(M\otimes X\right)$ some element of $M$ must occur infinitely many times. Let $m\in M$ be such an element and $(x_{n_{k}})$ be the subsequence of $(x_{n})$ such that $m_{n_k}=m$. Define $\gamma_{X} : \mathbf{C}\left(M\otimes X\right) \rightarrow M\otimes \left(\mathbf{C}(X)\right)$ by $~\gamma_X\left([m_n\otimes x_n]\right)=m\otimes [x_{n_{k}}] $. One can easily check that $\gamma_X$ is well defined and that it is inverse to $\beta_X$.\\

The collection of morphisms $\beta=\{\beta_X\}$ determines a natural isomorphism  from $F\circ \mathbf{C}$ to $\mathbf{C}\circ F$, whose inverse is $\gamma=\{\gamma_X\}$. The naturality of $\gamma$ and $\beta$ can be easily checked. Note that, $\gamma_X$'s and $\beta_X$'s are isometric embeddings and in particular short maps. Thus, $\mathbf{C}\circ F$ and $F\circ \mathbf{C}$ on $\ms$ are also naturally isomorphic. We record these observations in the following lemma.

\begin{lemma}[Lemma 13 of \cite{Bhat}]
$\bc\circ F$ and $F\circ \bc$ are naturally isomorphic as endofunctors on both $\mc$ and $\ms$
\end{lemma}

We may now apply the Cauchy completion process to the initial algebra $(G, g)$ of $F$ on $\set$. By composing $\beta_G$ with the Cauchy completion of $(G,g)$ we obtain an algebra $\bc(g)\circ \beta_G: M\otimes \bc(G)\rightarrow \bc(G)$. Let $S=\bc(G)$ and $s=\bc(g)\circ \beta_G$. As both $g$ and $\beta_G$ are isometric isomorphisms, it follows that $s$ is also an isometric isomorphism. In particular, $s: M\otimes S \rightarrow S$ is an isomorphism in both $\mc$ and $\ms$. From now onward, we will identify $s^{-1}$ by $\psi$. Thus, $~\psi=s^{-1}:S \rightarrow M\otimes S$ is a coalgebra for $F$. We will call both the $F$-algebra $(S, s)$ and the $F$-coalgebra $(S, \psi)$ the Cauchy completion of $(G, g)$. Moreover, throughout this paper we will consider $G$ as a dense subset of $S$, via the embedding $\nu_G$.

\begin{lemma}
$(S, \psi)$, the Cauchy completion of the initial algebra $(G, g)$ for $F$ on $\set$, is a coalgebra for $F$ on $\set$, $\mc$ and $\ms$. The structure map $\psi$ is an isometric isomorphism in $\mc$, and in particular an isomorphism in $\set$ as well as in $\ms$.
\end{lemma}

\begin{example}
\label{exmpSerpS}
It has been observed in \cite{Bhat}, that $(\s, \sigma)$ and $(S, \psi)$ are bi-Lipschitz isomorphic as $F$-coalgebras. Thus, $(\s, \sigma)$ is isomorphic to $(S, \psi)$, the Cauchy completion of $(G, g)$, as an $F$-coalgebra over $\set$ and $\mc$.
\end{example}

\section{ Final coalgebra for  $F$ on $\set$} 
\label{secOnSet}

It has already been observed in \cite{Bhat}, that $(S, \psi)$ is the final coalgebra for $F$ on $\set$. However, this section presents a different approach to establish this fact, which explicitly use the fact that $S$ is the Cauchy completion of $G$. It turns out that the metric structure of $G$, and hence that of $S$, is relevant even in the set setting. In particular, the method elicits why the Cauchy completion of $G$ is the final coalgebra (even in the set setting).  Moreover, the approach presented here enables us to establish $(S, \psi)$ as the final coalgebra for $F$ over $\mc$ in the next section, which is a stronger result than what has already been observed (for example in \cite{Bhat}).\\

Let $e: X \rightarrow M\otimes X = F(X)$ be a coalgebra for $F$ on $\set$, where $e$ is a  set function preserving distinguished points. By iterating this coalgebra, we obtain the following chain.
\begin{equation}
\nonumber
\label{coalgebraChain}
\begin{tikzcd}
X \arrow{r}{e} & M\otimes X\arrow{r}{M\otimes e} &  \cdots M^{n-1}\otimes X \arrow{r}{M^{n-1}\otimes e} & M^{n}\otimes X \arrow{r} & \cdots
\end{tikzcd}     
\end{equation}

For each $~x\in X$, we construct a sequence $\left(\chi_n \right)_{n\in \nn}$ such that $\chi_n = m_{1}\otimes m_{2}\otimes...\otimes m_n  \otimes x_{n} \in M^n\otimes X$ as follows. Let $\chi_0=x$. We define $\chi_n$ inductively by  $\chi_n = \ \left(M^{n-1} \otimes e \right) \left ( \chi_{n-1}\right) =\ \left( M^{n-1} \otimes e \right) \left(m_{1}\otimes m_{2}\otimes...\otimes m_{n-1}  \otimes x_{n-1}\right) = m_{1}\otimes m_{2}\otimes...\otimes m_n  \otimes x_{n}$. Here, $e(x_{n-1})=m_n\otimes x_n$.\\

For each sequence $\left(\chi_n \right)_{n\in \nn}$ associated with an $x\in X$, we define a corresponding sequence  $\left(\theta_{n}(x)\right)_{n\in\nn}$ in $G$ as follows. If $\chi_n = m_{1}\otimes m_{2}\otimes...\otimes m_n  \otimes x_{n}$, then we set $\theta_n(x) = [m_{1}\otimes m_{2}\otimes...\otimes m_n  \otimes \overline{x_{n}} ]$, where $\overline{x_{n}}$ is chosen to be $T$, $L$ or $R$ depending on whether $m_{n}$ is $a$, $b$  or $c$ respectively.\\

 Note that when defining the sequence $(\chi_n)$ we have chosen a particular representative $m_1\otimes \cdots \otimes m_n \otimes x_n$. Thus, one could wonder whether $\theta_{n}(x)$  depends on these particular choices we make. However it is not too hard to show that it is not the case.\\

Suppose we consider two choices for $\chi_n$, $\chi_n=m_{1}\otimes m_{2}\otimes\cdots\otimes m_n  \otimes x_{n}=m_{1}^{'}\otimes m_{2}^{'}\otimes\cdots\otimes m_{n}^{'}  \otimes y_{n}$. Let $\theta_n(x)$ and $\theta_n'(x)$ be the two corresponding choices.  For $n=1$, we have  $~m_1\otimes x_1=m_{1}'\otimes y_{1}~$. Without loss of generality let $m_{1}=a$. Then there are two possibilities for $m_{1}^{'}$, namely $b$ or $c$. Since $a\otimes L=b\otimes T$ and $a\otimes R=c\otimes T$, it follows that $\theta_1(x) = \theta_1'(x)$. Similarly we can show that $~\theta_n(x)=\theta_{n}^{'}(x)~$ for the other cases too. Thus, the corresponding sequences $~\theta_{n}(x)$ and $\theta_{n}^{'}(x)$ in $G$ are equal.

\begin{lemma}
\label{indpend}
$\left(\theta_{n}\left(x\right)\right)~$ is independent of the choice of representatives of $\left(\chi_n \right)$.
\end{lemma}

In the next lemma we show that the sequence $\left(\theta_n(x)\right)$ is a Cauchy sequence for every $x\in X$ and the choice of $N$ for any $\epsilon >0$ is independent of $x$.

\begin{lemma}
\label{cauchySeq}
For every $\epsilon >0$, there is $N\in \mathbb{N}$ such that for all $x\in X$ and $p, q> N$, $d_G\left(\theta_p(x) , \theta_q(x) \right)<\epsilon$. 

In particular, for every $x\in X$ the sequence $~(\theta_{n}(x))~$ is a Cauchy sequence in $G$.
\end{lemma}

\begin{proof}
Let $\epsilon>0\,$ and choose $\,N\in\mathbb{N}$ such that $\,2^{-N} <\epsilon$. For $\,q>p>N$ and $x\in X$, consider $d_{G}(\theta_{p}(x),\theta_{q}(x)  )=d_{G}([m_{1}\otimes \cdots m_{p} \otimes \overline{x_{p}}],[m_{1}\otimes  \cdots \otimes m_{q} \otimes \overline{x_{q}}]  )$. The right hand side of this is equal to  $~d_{G}([m_{1}\otimes \cdots m_{p}\otimes l_{p+1}\otimes\cdots  \otimes l_{q}\otimes \overline{y_{q}}],[m_{1}\otimes  \cdots \otimes m_{q} \otimes \overline{x_{q}}]  )$, where $~m_{1}\otimes \cdots m_{p}\otimes l_{p+1}\otimes\cdots  \otimes l_{q}\otimes \overline{y_{q}}~\sim~ m_{1}\otimes  \cdots \otimes m_{p} \otimes \overline{x_{p}}$. It now follows from (\ref{metricG}) that $~d_{G}([m_{1}\otimes \cdots m_{p}\otimes l_{p+1}\otimes\cdots  \otimes l_{q}\otimes \overline{y_{q}}],[m_{1}\otimes  \cdots \otimes m_{q} \otimes \overline{x_{q}}]  )$ is equal to $~d_{M^q\otimes I}(m_{1}\otimes \cdots m_{p}\otimes l_{p+1}\otimes\cdots  \otimes l_{q}\otimes \overline{y_{q}},m_{1}\otimes  \cdots \otimes m_{q} \otimes \overline{x_{q}}  )$ which is less than or equal to $\frac{1}{2^p}$ (by Lemma \ref{totalybound}). But  $\frac{1}{2^p}<\frac{1}{2^N}<\epsilon$. Thus we have $d_{G}(\theta_{p}(x),\theta_{q}(x)  )<\epsilon$.
\end{proof}

\begin{remark}
\label{uniCauchy}
Note that the statement of Lemma \ref{cauchySeq} means that the sequence of functions $\theta_n : X\rightarrow S$ is uniformly Cauchy. 
\end{remark}
 Thus one can now define a function $\,f:X\rightarrow S\,$ by considering $\lim_{n\rightarrow \infty}\theta_{n}(x)~$ in $S$.  $$f(x)=\displaystyle\lim_{n\rightarrow \infty}\theta_{n}(x).$$
It is here that we explicitly use the fact that $S$ is the Cauchy completion of $G$.\\

We claim that $f$ makes the square in (\ref{comSquareMead}) commute. It is not too hard to show that $f$ must be the only map from $X$ to $S$ for which the square (\ref{comSquareMead}) commutes.

\begin{equation}
\label{comSquareMead}
\begin{tikzcd}
X \arrow{r}{e} \arrow{d}{f} &  M\otimes X \arrow{d}{M\otimes f}\\
S \arrow{r}{\psi} & M\otimes S
\end{tikzcd}     
\end{equation}

\begin{proposition}
\label{defOfMediatingMor}
The function $f:X\rightarrow S$, given by $f(x)=\displaystyle\lim_{k\rightarrow \infty}\theta_{k}(x)$, is a morphism in $\set$ and it make the square in (\ref{comSquareMead}) commute. 

Moreover, it is the only morphism from $X$ to $S$ in $\set$, for which such a diagram commutes.
\end{proposition} 

\begin{proof}
We will prove the commutativity of the diagram as the other facts follow easily. 

We need to show that $\,\left( M\otimes f \right) \circ e=\psi \circ f$. For $\,x\in X\,$, $\chi_1=e(x)=m_{1}\otimes x_{1}\,,\, \chi_2=\left( M\otimes e\right) (m_{1}\otimes x_{1})=m_{1}\otimes m_{2}\otimes x_{2}, \cdots $. We have $f(x)=\displaystyle\lim_{k\rightarrow \infty}[m_{1}\otimes\cdots\otimes m_k  \otimes \overline{x_{k}}]$ and $f(x_{1})=\displaystyle\lim_{k\rightarrow \infty}[ m_{2}\otimes\cdots\otimes m_k  \otimes \overline{x_{k}}]$. Thus,
\begin{align*}
    \psi \left(f(x)\right) &=\psi \left(\displaystyle\lim_{k\rightarrow \infty}[m_{1}\otimes\cdots\otimes m_k  \otimes \overline{x_{k}}]\right)\\
    &=\displaystyle\lim_{k\rightarrow \infty}\psi([m_{1}\otimes\cdots \otimes \overline{x_{k}}])\\
    &=\displaystyle\lim_{k\rightarrow \infty}\left( m_{1}\otimes [m_{2}\otimes\cdots\otimes m_k  \otimes \overline{x_{k}}]\right).
\end{align*}
Note that  
\begin{align*}
    M\otimes f\left( e(x) \right)&=M\otimes f(m_{1}\otimes x_{1})=m_{1}\otimes f(x_{1})\\
    &=m_{1}\otimes \displaystyle\lim_{k\rightarrow \infty}[ m_{2}\otimes\cdots\otimes m_k  \otimes \overline{x_{k}}].
\end{align*}
But $m_{1}\otimes \displaystyle\lim_{k\rightarrow \infty}\theta_{k}(x)=\displaystyle\lim_{k\rightarrow \infty}m_{1}\otimes\theta_{k}(x)$. Thus the above diagram commutes.  
\end{proof} 

The following theorem now directly follows from the above proposition. The fact that $(S, \psi)$ is the final coalgebra for $F=M\otimes -$ on $\set$ has been observed in \cite{Bhat}, but the structure of the mediating morphism $f$ as described here is a new result. 

\begin{theorem}
\label{finOnSet}
 Let $(S, s)$ be the Cauchy completion of the initial algebra $(G, g)$ for $F$ on $\set$ with respect to the natural metric on $G$ and $\psi=s^{-1}$. Then $(S, \psi)$ is the final coalgebra for $F$ on $\set$.

Moreover, if $(X, e)$ is an $F$-coalgebra on $\set$, the mediating morphism $f:X\rightarrow S$ such that the square (\ref{comSquareMead}) commutes is given by $~f(x)=\displaystyle\lim_{k\rightarrow \infty}\theta_{k}~$.
\end{theorem}

It now follows from Example \ref{exmpSerpS} and Proposition \ref{finOnSet} that $(\s, \sigma)$ is also the final coalgebra for $F$ on $\set$. Thus, for any $F$-coalgebra $(X, e)$ there is a unique map $f:X\rightarrow \s$ preserving the distinguished points and a diagram similar to that of (\ref{comSquareMead}) with the bottom row replaced by $(\s, \sigma)$ commutes.

\begin{theorem}
$(\s, \sigma)$ is the final $F$-coalgebra in $\set$
\end{theorem}


\section{ Final coalgebra for  $~F~$ on $\mc$ and $\ms$} 
\label{secOnMSC}

We now turn in to the metric setting. We will first consider the continuous setting and then consider the short map setting.
\subsection{Final coalgebra for $F$ on $\mc$}
Our main observation here is that if $(X, e)$ is a coalgebra for $F$ on $\mc$, i.e., $e$ is continuous, then the mediating morphism $f$ from $X$ to $S$ defined in the previous section is also continuous. The following observation, which trivially follows from Lemma \ref{presFunctions}, is useful in the future.

\begin{lemma}
Let $\,I\,$ be the initial object and $\,X\,$ be any object in the category $\mc$. Let $\eta$ be the  unique morphism $\eta:I\longrightarrow X$ given by the initiality of $I$. Then $\eta$ and $M^{n}\otimes \eta$ are isometric embeddings for all $n\in \nn$.   
\end{lemma}

The next proposition guarantees the continuity of $f:X\rightarrow S$ in the setting of $\mc$.
\begin{proposition}
\label{contOff}
If $(X, e)$ is a coalgebra for $F$ on $\mc$ (i.e., $e$ is continuous), then the mediating morphism $f:X\rightarrow S$ (as in Theorem \ref{finOnSet}) is continuous.
\end{proposition} 

\begin{proof}
We will prove the continuity of $f$ at $x \in X$. Let $\,\epsilon >0\,$ be arbitrary. From Lemma \ref{cauchySeq}, there is  $N=N(\epsilon)$ such that for all $y\in X$ and $p\geq q> N$ $d_S(\theta_p(y), \theta_q(y))<\epsilon/8$. By taking the limit as $p\rightarrow \infty$ we have $$d_S(f(y), \theta_q(y))<\dfrac{\epsilon}{8}$$ for all $y\in X$ and $q>N$.\\

Let $q$ be such that $q>N$ and $\frac{1}{2^q}<\frac{\epsilon}{4}$.
Since $q>N$, it follows that for any $y\in X$
\begin{align*}
    d_S(f(x),f(y)) &\leq d_S(f(x), \theta_q(x))+d_S(\theta_q(x), \theta_q(y))+d_S(\theta_q(y), f(y))\\
    &< d_S(\theta_q(x), \theta_q(y)) +\dfrac{\epsilon}{4}.
\end{align*} 

Let $\,e^q=(M^{q-1}\otimes e)\circ \cdots \circ (M\otimes e)\circ e$ (similar to $e_s$ in (\ref{ek}) for an algebra). Because $e^q$ is continuous at $x$, there is $\delta=\delta(\epsilon) >0$ such that 
 $~d_{M^{q}\otimes X }(e^{q}(x),e^q(y))<\frac{\epsilon}{4}$ whenever $d(x,y)<\delta$.\\
 
Let $y\in X$ such that $d(x,y)<\delta$. Consider the two sequences $\left(\theta_k(x)\right)$ and $\left(\theta_p(y)\right)$. If $e^q(x)=m_{1}\otimes \cdots m_{q}\otimes {x_{q}}$ and $e^q(y)=n_{1}\otimes \cdots n_{q}\otimes {y_{q}}$, then $\displaystyle \theta_q(x)= \left[ m_{1}\otimes \cdots m_{q}\otimes \overline{x_{q}} \right] $ and $\displaystyle \theta_q(y)= \left[n_{1}\otimes \cdots n_{q}\otimes \overline{y_{q}} \right]$. We now have,
\begin{eqnarray*}
d_{S}(f(x),f(y)) & < \dfrac{\epsilon}{4}~~  + & d_{S}\left([m_{1}\otimes \cdots m_{q}\otimes \overline{x_{q}}],[n_{1}\otimes \cdots n_{q}\otimes \overline{y_{q}}]\right) \\
 & = \dfrac{\epsilon}{4}~~  + & d_{M^{q}\otimes I}\{m_{1}\otimes \cdots m_{q}\otimes \overline{x_{q}},n_{1}\otimes \cdots n_{q}\otimes \overline{y_{q}}\}\\ 
 & &~~~~~~~~~\lceil~\because \, M^{q}\otimes I\xrightarrow{C_q}G \textrm{ is an isometric embedding }\rfloor\\
& =\dfrac{\epsilon}{4}~~  + & d_{M^{q}\otimes X}\{m_{1}\otimes \cdots m_{q}\otimes \overline{x_{q}},n_{1}\otimes \cdots n_{q}\otimes \overline{y_{q}}\}\\ 
 & &\lceil\because \, M^{q}\otimes I\xrightarrow{M^{q}\otimes \eta} M^{q}\otimes X\textrm{is an isometric embedding}\rfloor\\
 &\leq \dfrac{\epsilon}{4}~~  + & d_{M^{q}\otimes X}\{m_{1}\otimes \cdots m_{q}\otimes \overline{x_{q}},m_{1}\otimes \cdots m_{q}\otimes x_{q}\}\\
 & & + ~~ d_{M^{q}\otimes X}\{m_{1}\otimes \cdots m_{q}\otimes x_{q},n_{1}\otimes \cdots n_{q}\otimes y_{q}\}\\
 & & + ~~ d_{M^{q}\otimes X}\{n_{1}\otimes \cdots n_{q}\otimes y_{q},n_{1}\otimes \cdots n_{q}\otimes  \overline{y_{q}}\}\\
 & \leq  \dfrac{\epsilon}{4}~~  + & \frac{1}{2^{q}}+d_{M^{q}\otimes X}\{e^{q}(x),e^{q}(y)\}+\frac{1}{2^{q}}~~~~~~~~\lceil~\because \, \textrm{ Lemma } \ref{totalybound}~\rfloor\\ 
 & <  \dfrac{\epsilon}{4} ~~  + &\frac{\epsilon}{4}~~+~~\frac{\epsilon}{4}~~+~~\frac{\epsilon}{4}~~=~~\epsilon
\end{eqnarray*}

Thus, we have the continuity of $f$ at $x$, completing the proof.	 
\end{proof}

\begin{remark}
\begin{enumerate}
    \item As $S$ is complete, it follows from Remark \ref{uniCauchy}, that the sequence of functions $(\theta_n)$ converge uniformly to the function $f$
    \item Each $\theta_n$ is continuous. One can extract a proof of this from the proof of Proposition \ref{contOff}.
    \item Then, it follows that $f$ is continuous as the uniform limit of continuous functions is continuous.
\end{enumerate}
\end{remark}

The following result is a straight forward implication of Proposition \ref{contOff}.

\begin{theorem}
$(S, \psi)$ is the final coalgebra for $F$ on  $\mc$.
\end{theorem}

Recall from Example \ref{exmpSerpS} that $\s$ is isomorphic to $S$ as $F$-coalgebras on $\mc$. Thus we have the following theorem, which is one of the main claims of this paper.

\begin{theorem}
\label{SisFinalinMC}
$(\s, \sigma)$ is the final coalgebra for $F$ on $\mc$
\end{theorem}

\subsection{Final coalgebra for $F$ on $\ms$}

It has been observed in \cite{Bhat} that $(S, \psi)$ is the final coalgebra for $F$ on $\ms$. Our main goal here is to see how this observation follows from the above discussion. We need to show that if $(X, e)$ is a coalgebra with $e$ being a short map, then the mediating morphism $f:X\rightarrow S$ is also a short map. This follows from the proof of Proposition \ref{contOff} with some minor modifications.

\begin{proposition}
\label{shortOff}
If $(X, e)$ is a coalgebra for $F$ on $\ms$ (i.e., $e$ is a short map), then the mediating morphism $~f:X\rightarrow S$ is also a short map.
\end{proposition} 

\begin{proof}
We give the required modifications to the proof of Proposition \ref{contOff}.

Let $x,y\in X$. Let $\,\epsilon >0\,$ be arbitrary. Then there is some $q$ such that $2^{-q}<\frac{\epsilon}{3}$ and $ d_S(f(x),f(y)) < d_S(\theta_q(x), \theta_q(y)) +\frac{\epsilon}{3}$.

Note that $\,e^q=(M^{q-1}\otimes e)\circ \cdots \circ (M\otimes e)\circ e$ is also a short map (Lemma \ref{presFunctions}) and hence  $~d_{M^{q}\otimes X }(e^{q}(x),e^q(y))<d(x,y)$.\\
 
Consider the two sequences $\left(\theta_q(x)\right)$ and $\left(\theta_q(y)\right)$. If $e^q(x)=m_{1}\otimes \cdots m_{q}\otimes {x_{q}}$ and $e^q(y)=n_{1}\otimes \cdots n_{q}\otimes {y_{q}}$, then $\displaystyle \theta_q(x)= \left[ m_{1}\otimes \cdots m_{q}\otimes \overline{x_{q}} \right] $ and $\displaystyle \theta_q(y)= \left[n_{1}\otimes \cdots n_{q}\otimes \overline{y_{q}} \right]$. As in Proposition \ref{contOff}, we now have,
\begin{eqnarray*}
d_{S}(f(x),f(y)) & \leq  \dfrac{\epsilon}{3}~~  + & \frac{1}{2^{q}}+d_{M^{q}\otimes X}\{e^{q}(x),e^{q}(y)\}+\frac{1}{2^{q}}\\ 
 & <  \dfrac{\epsilon}{3} ~~  + &\frac{\epsilon}{3}~~+~~d(x,y)~~+~~\frac{\epsilon}{3}~~=~~\epsilon~~+~~d(x,y).
\end{eqnarray*}

Since $\epsilon$ is arbitrary, $d_{S}(f(x),f(y))\leq d(x,y)$. Hence, $f$ is a short map. 	 
\end{proof}

\begin{theorem}[Theorem 16 of \cite{Bhat}]
$(S, \psi)$ is the final coalgebra for $F$ on  $\ms$.
\end{theorem}

However, as the unique coalgebra map from $(S, \psi)$ to $(\s, \sigma)$ is not short (and only Lipschitz) it follows that $(\s, \sigma)$ is not the final coalgebra for $F$ on $\ms$


\section{Examples} 
\label{examples}

\begin{example}$(G,g)$ is not the initial algebra in $\mc$. \emph{Consider $I=\{T,L,R\}$, the initial object in $\mc$, and the algebra $e:M\otimes I\rightarrow I$ given by $e(a\otimes T)=T$, $e(a\otimes L=b\otimes T)=e(b\otimes L)=L$ and $e(b\otimes R =c\otimes L)=e(a\otimes R=c\otimes T)=e(c\otimes R)=R$.
} \end{example}

Clearly $~e~$ is Lipschitz as it is a function from a finite metric space and hence it is continuous. As $(G,g)$ is the initial algebra in $\set$, there is a unique function $h:G\rightarrow I$ in $\set$ such that $e\circ (M\otimes h)=h\circ g$.\\

Let $b^n\otimes d=b\otimes \cdots \otimes b\otimes d \in M^{n}\otimes I$, where $b$ occurs $n$ times and $d\in I$. We can compute the image of $[b^n\otimes d]$ under $h$ by induction.\\

For example, consider $[b^n\otimes R]\in G$. When $n=1$,
\begin{multline*}
      h\left(\left[b\otimes R\right]\right)=h\left(g\left(b\otimes\left[ R \right]\right)\right) = e\left(\left(M\otimes h\right)\left(b\otimes \left[R\right]\right)\right)  =e\left(b\otimes R\right)=R.
\end{multline*}
For $n>1$, assuming $h\left(\left[b^{n-1}\otimes R\right]\right)=R$, we have,
\begin{multline*}
    h\left(\left[b^n\otimes R\right]\right)=h\left(g\left(b\otimes\left[ b^{n-1}\otimes R \right]\right)\right) \\
    = e\left(\left(M\otimes h\right)\left(b\otimes \left[ b^{n-1}\otimes R \right]\right)\right)=e\left(b\otimes R\right)=R.
\end{multline*}
Thus, for every $n\in \nn$, $h\left(\left[b^n\otimes R\right]\right)=R$.\\

Now consider $b^n\otimes L$. In this case, we may do without induction. Note that in the initial chain (\ref{initialChain}), $!(L)=b\otimes L$. Thus,  $(M^{n}\otimes !)(b^n\otimes L)=b^{n+1}\otimes L$.  Therefore, $[b^n\otimes L]=[L]\in G$. In particular, $h\left([b^n\otimes L]\right)=h\left([L]\right)=L$.\\

It now follows from Lemma \ref{totalybound} that 
$$d_G\left(\left[b^n\otimes R \right],\left[ L \right] \right)=d_G\left(\left[b^n\otimes R \right],\left[ b^n\otimes L \right] \right) \leq 2^{-n}.$$
However, $$ d\left( h\left(\left[b^n\otimes R \right]\right),h\left(\left[ L \right] \right) \right) = d\left(R,L\right)=1 $$
Thus, $h$ is not continuous at $[L]$. Hence, $(G, g)$ cannot be the initial algebra for $F$ in $\mc$. We do not know whether the initial algebra for $F$ on $\mc$ exist or not.

\begin{example}$(S,\psi)$ is not the final coalgebra in the Lipschitz setting.
\emph{
Consider the tri-pointed set $\triangle=\{(x,0) ~:~ x\in [0~,~1]\}\cup\{(\frac{1}{2},  \frac{\sqrt{3}}{2}) \}$, whose distinguished elements are $T_{\triangle}=(\frac{1}{2},  \frac{\sqrt{3}}{2})$, $L_{\triangle}=(0,  0)$ and $R_{\triangle}=(1,  0)$, and whose metric is the euclidean metric on $\mathbb{R}^{2}$.}
\end{example}
Let $~e^{'}:\triangle\longrightarrow M\otimes \triangle~$ be given by 
$$e^{'}(x,y)=\left\{%
\begin{array}{lllll}
a\otimes (\frac{1}{2},  \frac{\sqrt{3}}{2}), &  \hbox{$(x,y)=(\frac{1}{2},  \frac{\sqrt{3}}{2})$;}\\\\
b\otimes (0,0), & \hbox{$x\in [0~\frac{1}{4}] ~\&~ y=0 $;} \\\\
b\otimes (4x-1,0), & \hbox{$x\in [\frac{1}{4}~\frac{1}{2}]~\&~ y=0 $;}\\\\
c\otimes (4x-2,0), & \hbox{$x\in [\frac{1}{2}~\frac{3}{4}]~\&~ y=0 $;}\\\\
c\otimes (1,0), & \hbox{$x\in [\frac{3}{4}~1]~\&~ y=0$.}
\end{array}%
\right.$$

Then, $~e^{'}~$ is a Lipschitz map with Lipschitz constant 2 and hence  $~(e^{'},\triangle)~$ is an $F$-coalgebra in $\mc$.\\ 

As $(\s, \sigma)$ is the final coalgebra for $F$ on $\mc$ (Theorem \ref{SisFinalinMC}), there is a unique continuous map $f:(\triangle~,~e')\rightarrow (\s,\sigma)~$ such that the following diagram commutes.
\begin{equation*}
\begin{tikzcd}
\triangle \arrow{d}{f} \arrow{r}{e'} & M\otimes \triangle \arrow{d}{M\otimes f}\\
\s \arrow{r}{\sigma}  &  M\otimes\s
\end{tikzcd}     
\end{equation*}
Then, it follows that $~f~$ must satisfy the following conditions.\\\\
$~f(x,0)=0 , ~ x \in [0~\frac{1}{4}]~~,~~f(x,0)=1,~x\in [\frac{3}{4}~1]~$ and $$f(x,0)=\left\{%
\begin{array}{ll}
\displaystyle\frac{f(4x-1,0)}{2}, & \hbox{$x\in [\frac{1}{4}~\frac{1}{2}] $;} \\\\
\displaystyle\frac{1+f(4x-2,0)}{2}, & \hbox{$x\in [\frac{3}{4}~1]$.}
\end{array}%
\right.$$
Define the following families of intervals, where $n \in \nn$. 
\begin{align*}
I_n & = \left[\displaystyle \frac{1}{4}~,~\displaystyle\frac{1}{4}+\cdots+\frac{1}{4^{n}}\right]\\ 
J_n & = \left[\displaystyle \frac{1}{4}+\cdots+\frac{1}{4^{n}}+\frac{3}{4^{n+1}}~,~\displaystyle\frac{1}{4}+\cdots+\frac{1}{4^{n}}+\frac{4}{4^{n+1}}\right]
\end{align*}
Then, it can be shown by induction that $f$ satisfy the following properties (We leave the verification to the reader).
\begin{align*}
	f(x,0)&=0~, x\in I_{n}~,\forall~n\in\mathbb{N}\\
	f(x,0)&=\displaystyle\frac{1}{2^{n}}~,x\in J_{n}~,\forall~n\in\mathbb{N}
\end{align*}
 With these two properties, we can now show that $f$ is not Lipschitz. Suppose to the contrary that $f$ is Lipschitz. Then we have some $k>0$ such that $~d(f(x,0),f(y,0))\leq k\cdot d\left((x,0),(y,0)\right),\,\,~\forall\, x,y \in [0,1]$.  Choose $\displaystyle x= \frac{1}{4}+\cdots+\frac{1}{4^{n}}+\frac{1}{4^{n+1}}~$ and $\displaystyle y=\frac{1}{4}+\cdots+\frac{1}{4^{n}}+\frac{3}{4^{n+1}}$. Then $f(x,0)=0$ and $\displaystyle f(y,0)= \frac{1}{2^{n}}$. From the Lipschitz condition, we have $\displaystyle \frac{1}{2^{n}}\leq k \frac{2}{4^{n+1}}, ~\forall\, n\in \mathbb{N}$; which implies 
 $k\geq 2.2^{n}, ~\forall\, n\in \mathbb{N}$. Hence $k$ is not bounded, which is a contradiction. Therefore $f$ is not Lipschitz.\\
 
 Hence, neither $(S,\psi)$ nor $(\mathbb{S},\sigma)$ can be the final coalgebra for $F$ in the Lipschitz setting. It was questioned in \cite{Bhat}, whether $(S,\psi)$, the final coalgebra for $F$ on $\ms$ is the final coalgebra for $F$ on $\ml$. This example shows that it is not the case. 


 \section*{Acknowledgement}
 The authors respectfully acknowledge the guidance and resourcefulness of  Professor Lawrence S. Moss.



\section*{References}
\bibliographystyle{model1-num-names}

\end{document}